\documentclass[12pt]{amsart}

\usepackage[colorlinks,
linkcolor=blue,
anchorcolor=blue,
citecolor=blue]{hyperref}
\usepackage{epsfig}
\usepackage{graphicx}
\usepackage{amssymb, amstext, amscd, amsmath}
\usepackage{amsthm, mathrsfs, amsfonts,dsfont}
\usepackage{fullpage}
\usepackage{txfonts}
\usepackage{fancybox}
\usepackage{color}
\usepackage{cite}
\usepackage{comment}
\usepackage{tikz-cd} 

\allowdisplaybreaks

\newtheorem{theorem}{Theorem}[section]
\newtheorem{lemma}{Lemma}[section]

\newtheorem{corollary}{Corollary}[section]

\theoremstyle{definition}

\newtheorem{remark}{Remark}[section]

\numberwithin{equation}{section}

\begin{document}
\title {Strict Monotonicity of Numerical Invariants for the Submodules $[(z-w)^k]$ in $H^2(\mathbb D^2)$}
\author{Yin Liu}
\address{School of Mathematics and Statistics, Nanyang Normal University,
Nanyang, Henan, 473061, P. R. China} 
\email{lylight@mail.bnu.edu.cn}

\author{Yufeng Lu}
\address{School of Mathematical Sciences, Dalian University of Technology,
Dalian, Liaoning, 116024, P. R. China}
\email{lyfdlut@dlut.edu.cn}

\author{Chao Zu*}
\address{School of Mathematical Sciences, Dalian University of Technology,
Dalian, Liaoning, 116024, P. R. China}
\email{zuchao@mail.dlut.edu.cn}

\subjclass[2020]{Primary 46E22; Secondary 47B32, 47A13}
\thanks{*Corresponding author.}
\keywords{Hardy module over the bidisk; homogeneous submodule; numerical invariant; core operator; Toeplitz matrix.}

\begin{abstract}
For $k\geq1$, let $M_k=[(z-w)^k]\subset H^2(\mathbb D^2)$. We first determine the banded Toeplitz
matrices associated with the homogeneous components of $M_k$, together
with explicit formulas for their determinants and the relevant
algebraic cofactors. These formulas lead to a complete description of
the spectrum of the core operator:
\[
\sigma(C_{M_k})
=
\{0,1\}
\cup
\left\{
\pm\frac{k}{n+k}:n\geq1
\right\}.
\]
In particular, the spectral data determine the parameter $k$.

The determinant and cofactor formulas further yield a unified
finite-sum representation for $\alpha_{n,j}^{(k)}
=\langle w^j\phi_n,z^j\psi_n\rangle$,
and hence for Yang's higher numerical invariants. We derive an adjacent
relation connecting $\alpha_{n,j}^{(k)}$ and
$\alpha_{n,j+1}^{(k)}$ by means of an explicit telescoping
certificate, and show that the corresponding finite-section
transformations are strict contractions. Combining these finite-dimensional estimates with the asymptotic behavior of
$\alpha_{n,j}^{(k)}$, we prove the strict monotonicity
\[
\Sigma_0(M_k)>
\Sigma_1(M_k)>
\Sigma_2(M_k)>
\cdots .
\]
The cases $k\geq3$ constitute the new part of the analysis, while the
previously known cases $k=1,2$ are recovered within the same framework.
Consequently, Yang's monotonicity conjecture holds in strict form for
the entire family $\{[(z-w)^k]:k\geq1\}$.
\end{abstract}



\maketitle

\section{Introduction}
\label{sec1}

Let $H^2(\mathbb D^2)$ be the Hardy space over the unit bidisk $\mathbb D^2 =\{(z,w)\in\mathbb C^2:|z|<1,\ |w|<1\}$,
viewed as a Hilbert module over $\mathbb C[z,w]$.
A closed subspace $M\subseteq H^2(\mathbb D^2)$ is called a submodule
if $zM\subseteq M$ and $wM\subseteq M$.
Although the invariant-subspace structure of the one-variable Hardy space
is completely described by Beurling's theorem \cite{Beu},
the corresponding problem over the bidisk is substantially more complicated.
Even principal submodules generated by elementary homogeneous polynomials
may exhibit nontrivial operator-theoretic and module-theoretic phenomena;
see, for example,
\cite{Che,Dou,Rud,Survey}.

From the operator-theoretic point of view, a submodule
$M\subset H^2(\mathbb D^2)$ naturally gives rise to the restriction
pair on $M$ and the compression pair on the quotient module
$H^2(\mathbb D^2)/M$. These operator pairs provide a basic framework
for studying the structure and unitary equivalence of submodules over
the bidisk; see
\cite{Dou2,Ya3,Ya2,Core}.

A central object in this framework is the core operator $C_M$.
It is a bounded self-adjoint operator whose unitary-equivalence class
provides operator-theoretic information about the submodule
\cite{Ya2005}. Through the two defect spaces $M\ominus zM$ and $M\ominus wM$, the core operator is closely related to Yang's numerical invariants
\cite{Fat}.

The Hilbert--Schmidt properties of submodules over the bidisk have
also been investigated for several important classes of finitely
generated and inner-function-related submodules; see, for example,
\cite{Luo,Zou,
Zu1,Zu2}.


More generally, if $\{\phi_n:n\geq0\}$ and $\{\psi_n:n\geq0\}$ are homogeneous orthonormal bases of $M\ominus zM$ and
$M\ominus wM$, respectively, Yang's higher numerical invariants are
given by
\[
\Sigma_j(M)
=
\sum_{n=0}^{\infty}
\left|
\left\langle w^j\phi_n,z^j\psi_n\right\rangle
\right|^2,
\qquad j\geq0;
\]
see \cite{Ya3}.
Yang conjectured that the sequence $\{\Sigma_j(M)\}_{j\geq0}$
is nonincreasing for every submodule of $H^2(\mathbb D^2)$.
The conjecture remains open in general, and explicit computations for
concrete homogeneous submodules therefore provide useful tests of this
operator-theoretic phenomenon
\cite{Fat,Liu,Liu2}.

In the present paper, we study the family of principal homogeneous submodules
\[
M_k=[(z-w)^k],\qquad k\geq1.
\]
The cases $k=1$ and $k=2$ were studied previously (see \cite{Ya3,Fat,Liu}), whereas the new
cases considered here are $k\geq3$. Nevertheless, the arguments developed
below apply uniformly to all $k\geq1$. This allows us not only to treat
all higher powers $(z-w)^k$, $k\geq3$, but also to recover the known
cases $k=1,2$ within the same framework.

The generator $(z-w)^k$ contains a repeated linear factor, and this
multiplicity produces a matrix structure quite different from that of
polynomials with distinct linear factors. In particular, the Toeplitz
matrices arising from the homogeneous wandering spaces are no longer
sparse residue-class matrices. Instead, they are real symmetric $k$-banded Toeplitz matrices whose diagonals are determined by the binomial coefficients of $|1-\zeta|^{2k}$.
This banded structure is the starting point of our analysis.

More precisely, we first determine the banded Toeplitz matrices
associated with the homogeneous components of $M_k$. We then obtain
closed formulas for their determinants and for the algebraic cofactors
in the first and last rows. These formulas have two consequences.

First, they allow us to determine the complete spectral data of the
core operator. In Theorem~\ref{thm:core-spectrum}, we prove that
\[
\sigma(C_{M_k})
=
\{0,1\}\cup
\left\{
\pm\frac{k}{n+k}:n\geq1
\right\}.
\]
In particular, the largest nontrivial positive eigenvalue is $\frac{k}{k+1}$,
so the core spectrum itself determines the parameter $k$.

Second, the determinant and cofactor formulas lead to a unified
finite-sum representation for
\[
\alpha_{n,j}^{(k)}
=
\langle w^j\phi_n,z^j\psi_n\rangle,
\qquad n\geq0,\quad j\geq1.
\]

A central feature of the present problem is the adjacent relation
established in Theorem~\ref{thm:adjacent-relation}:
\[
\begin{aligned}
&(n+k+j+2)\alpha_{n+1,j+1}^{(k)}=(n-j+k+1)\alpha_{n,j+1}^{(k)}
+(n+k+j)\alpha_{n,j}^{(k)}-(n-j+k+1)\alpha_{n+1,j}^{(k)}.
\end{aligned}
\]
We prove this identity directly from the finite-sum formula by means
of an explicit telescoping certificate. The adjacent relation can then
be written as a finite-dimensional matrix transformation between the
coefficient vectors corresponding to $j$ and $j+1$.

The main structural result, Theorem~\ref{thm:finite-section},
shows that these finite-dimensional transformations are strict
contractions. Consequently, for every
$j\geq1$ and every admissible $N$,
\[
\sum_{n=0}^{N}
\left|\alpha_{n,j}^{(k)}\right|^2
>
\sum_{n=0}^{N}
\left|\alpha_{n,j+1}^{(k)}\right|^2,
\qquad
N\geq\max\{0,j-k\}.
\]
Thus the monotonicity phenomenon is already present at the level of
finite homogeneous sections; it is not merely a consequence of
cancellation in the infinite series.

To preserve strictness when passing to the infinite sums, we also
determine the large-$n$ behavior
\[
\alpha_{n,j}^{(k)}
=
-\frac{k}{n+k}
+
\frac{kj}{(n+k)^2}
+
O(n^{-3}).
\]
It follows that
\[
\left|\alpha_{n,j}^{(k)}\right|^2
-
\left|\alpha_{n,j+1}^{(k)}\right|^2
=
\frac{2k^2}{(n+k)^3}
+
O(n^{-4}),
\]
so the termwise difference is eventually positive.

Our main theorem is the following.
\begin{theorem}
\label{thm:main-introduction}
Let $M_k=[(z-w)^k], k\geq1$. Then
\[
\Sigma_0(M_k)
=
k^2\sum_{m=k}^{\infty}\frac{1}{m^2},
\qquad
\Sigma_1(M_k)=\Sigma_0(M_k)-1,
\]
and
\[
\Sigma_0(M_k)>
\Sigma_1(M_k)>
\Sigma_2(M_k)>
\cdots .
\]
Consequently, Yang's monotonicity conjecture holds in the strict form
for the family $\{[(z-w)^k]:k\geq1\}$.
\end{theorem}

The cases $k=1$ and $k=2$ are consistent with the previously known
results. The new content of Theorem~\ref{thm:main-introduction}
is the uniform treatment of all powers $k\geq3$.

The passage from $k=2$ to arbitrary $k$ is not a formal extension:
the associated Toeplitz matrices have increasing bandwidth, and the
proof of monotonicity requires a new finite-sum identity, an adjacent
recurrence, and a family of finite-section contraction arguments that
work uniformly for all $k$.


We emphasize that the mechanism behind the present results is specific
to the repeated-factor structure of $(z-w)^k$. For comparison,
\[
z^k-w^k
=
\prod_{\omega^k=1}(z-\omega w)
\]
has distinct linear factors and gives rise to a residue-class
decomposition of the associated Toeplitz matrices. In the present
setting, no such decomposition is available; instead, the matrices are
genuinely $k$-banded, and the strict decrease of the numerical
invariants is governed by the adjacent relation and the resulting
finite-section contractions.

The two families also exhibit different spectral behavior. For the
submodules $[z^k-w^k]$, the nonzero spectrum of the core operator is
independent of $k$ \cite{Liu2}, whereas for $M_k=[(z-w)^k]$, the core spectrum contains
\[
\pm\frac{k}{n+k},
\qquad n\geq1,
\]
and therefore determines $k$. Thus the multiplicity of the repeated
factor $z-w$ is already reflected in the spectral data of the core
operator.

The paper is organized as follows. In Section~\ref{sec2}, we introduce the
operator-theoretic notation and recall the basic results on homogeneous
principal submodules, core operators, defect spaces, and numerical
invariants that will be used later. In Section~\ref{sec3}, we determine the
banded Toeplitz matrices associated with $(z-w)^k$, compute their
determinants and the relevant algebraic cofactors, and use these
formulas to determine the complete spectrum of the core operator.
In Section~\ref{sec4}, we derive a unified expression for $\alpha_{n,j}^{(k)}$, establish the adjacent relation by an explicit
telescoping argument, and prove the finite-section contraction.
Finally, the asymptotic behavior of $\alpha_{n,j}^{(k)}$ is used to deduce the strict monotonicity of the complete numerical invariant
sequence.

\section{Preliminaries}
\label{sec2}
In this section, we fix the operator-theoretic notation and recall the
basic facts on homogeneous principal submodules, defect spaces, and
numerical invariants that will be used throughout the paper.

\subsection{Restriction pairs, compression pairs, and the core operator}

Let $T_z$ and $T_w$ denote multiplication by the coordinate functions
$z$ and $w$ on $H^2(\mathbb D^2)$, respectively. If
$M\subseteq H^2(\mathbb D^2)$ is a submodule, then the restriction pair
is defined by $R_z=T_z|_M,
\,\,R_w=T_w|_M$. Since $M$ is invariant under both coordinate multipliers,
$(R_z,R_w)$ is a pair of commuting isometries on $M$.

Identifying the quotient module $H^2(\mathbb D^2)/M$ with
$M^\perp=H^2(\mathbb D^2)\ominus M$, we also consider the compression pair
$S_z=P_{M^\perp}T_z|_{M^\perp},\,\,S_w=P_{M^\perp}T_w|_{M^\perp}$, which is a pair of commuting contractions. Thus the restriction pair and
the compression pair describe, respectively, the operator-theoretic
structures of the submodule and the corresponding quotient module;
see \cite{Dou2,Ya3,Survey}.

A central operator associated with $M$ is the core operator $C_M$.
It is a bounded self-adjoint operator and vanishes on $M^\perp=H^2(\mathbb D^2)\ominus M$; see \cite{Core,Fat}. Hence the nonzero spectral data of $C_M$ are
determined by its restriction to $M$. On $M$, the core operator admits
the representation
\[
C_M
=
I-R_zR_z^*-R_wR_w^*
+R_zR_wR_z^*R_w^*.
\]

Since $R_z$ and $R_w$ are isometries, the self-commutators
\[
P_z:=[R_z^*,R_z]
=
I-R_zR_z^*,
\qquad
P_w:=[R_w^*,R_w]
=
I-R_wR_w^*
\]
are the orthogonal projections onto the defect spaces $M\ominus zM$ and $M\ominus wM$ respectively.

For submodules generated by finitely many polynomials, the operators $T_M:=P_zP_wP_z$ and $S_M:=[R_z^*,R_w][R_w^*,R_z]$
are trace class, and the first two numerical invariants are defined by
\[
\Sigma_0(M)=\operatorname{tr}(T_M),
\qquad
\Sigma_1(M)=\operatorname{tr}(S_M);
\]
see \cite{Ya4,Ya3}. Moreover,
$C_M^2\simeq T_M\oplus S_M$,
where $\simeq$ denotes unitary equivalence. Consequently, for a
submodule generated by finitely many polynomials,
\[
\operatorname{tr}(C_M^2)
=
\Sigma_0(M)+\Sigma_1(M)<\infty.
\]
Since $C_M$ is self-adjoint, it follows that $C_M$ is
Hilbert--Schmidt and hence compact. In particular,
\[
\|C_M\|_{\mathrm{H.S.}}^2
=
\operatorname{tr}(C_M^2)
=
\Sigma_0(M)+\Sigma_1(M).
\]
These relations explain the operator-theoretic origin of the numerical
invariants and will also be used in the spectral analysis below; see
\cite{Ya2005,Ya1}.

\subsection{Homogeneous principal submodules and Toeplitz matrices}
Let
\[
p(z,w)
=
\sum_{r=0}^{d}c_rz^rw^{d-r},
\qquad c_r\in\mathbb R,
\]
be a nonzero homogeneous polynomial of degree $d$, and let $M=[p]$ denote the principal submodule generated by $p$.
The restriction to real coefficients is sufficient for the present
paper, since our generator is
\[
p_k(z,w)=(z-w)^k.
\]

For $m\geq0$, let $H_m
=\operatorname{span}
\{z^rw^{m-r}:0\leq r\leq m\}$
be the space of homogeneous polynomials of degree $m$. Then
\[
M
=
\bigoplus_{m=0}^{\infty}pH_m.
\]
Consequently,
\[
M\ominus zM
=
\mathbb Cp
\oplus
\bigoplus_{m=1}^{\infty}
\left(
pH_m\ominus pzH_{m-1}
\right),
\]
and similarly,
\[
M\ominus wM
=
\mathbb Cp
\oplus
\bigoplus_{m=1}^{\infty}
\left(
pH_m\ominus pwH_{m-1}
\right).
\]
This graded decomposition provides the starting point for the explicit
construction of orthonormal bases of the two defect spaces; see
\cite[Section~2]{Fat}.

For $n\geq0$, consider the homogeneous vectors
$pw^n,\;
pzw^{n-1},\;
\ldots,\;
pz^n$. Their Gram matrix is denoted by $A^n=(a_{i,j})_{i,j=0}^{n}$,
where $a_{i,j}=
\left\langle
pz^iw^{n-i},
pz^jw^{n-j}
\right\rangle$.

For $m\geq0$, set $\gamma_m
=\langle pw^m,pz^m\rangle$. Since $p$ has real coefficients, each $\gamma_m$ is real. We therefore extend the notation to negative indices by setting $\gamma_{-m}=\gamma_m, m\geq1$.

Moreover, $a_{i,j}=
\gamma_{i-j},0\leq i,j\leq n$. Therefore $A^n$ is a real symmetric Toeplitz matrix of the form
\[
A^n=
\begin{pmatrix}
\gamma_0
& \gamma_1
& \gamma_2
& \cdots
& \gamma_n
\\
\gamma_1
& \gamma_0
& \gamma_1
& \cdots
& \gamma_{n-1}
\\
\gamma_2
& \gamma_1
& \gamma_0
& \cdots
& \gamma_{n-2}
\\
\vdots
& \vdots
& \vdots
& \ddots
& \vdots
\\
\gamma_n
& \gamma_{n-1}
& \gamma_{n-2}
& \cdots
& \gamma_0
\end{pmatrix}.
\]

The vectors $pw^n, pzw^{n-1},\ldots, pz^n$ are linearly independent. 
Therefore $A^n$ is positive
definite and, in particular, invertible.

For $n\geq1$, we set $D_n=\det A^{n-1}, D_0=1$.
Thus, $\det A^n=D_{n+1}$.

For $0\leq i,j\leq n$, let $A^n_{i,j}$ denote the algebraic cofactor
of the entry in row $i$ and column $j$ of $A^n$. Thus
$A^n_{i,j}=
(-1)^{i+j}\det A^n(i|j)$, where $A^n(i|j)$ denotes the matrix obtained from $A^n$ by deleting row $i$ and column $j$. For $n=0$, we use the standard convention that the determinant of the
empty matrix is $1$, so that
$A^0_{0,0}=1$.

The adjugate identity gives
\begin{equation}
\label{eq:inverse-cofactor}
\bigl((A^n)^{-1}\bigr)_{i,j}
=
\frac{A^n_{j,i}}{D_{n+1}}.
\end{equation}

\subsection{Defect-space bases and numerical invariants}

The following construction of orthonormal bases for the two defect
spaces is essentially Proposition~2.1 of
\cite{Fat}.

\begin{lemma}
\label{lem:defect-bases}
Let $M=[p]$, where $p$ is a nonzero homogeneous polynomial with real
coefficients. Then $\phi_0=\psi_0=\frac{p}{\|p\|}$, and for every $n\geq1$, we have
\[
\phi_n
=
\frac{
\displaystyle
\sum_{j=0}^{n}
pA^n_{0,j}z^jw^{n-j}
}{
\sqrt{D_nD_{n+1}}
},\,\,\,\,\qquad
\psi_n
=
\frac{
\displaystyle
\sum_{j=0}^{n}
pA^n_{n,j}z^jw^{n-j}
}{
\sqrt{D_nD_{n+1}}
}.
\]
Moreover, $\{\phi_n:n\geq0\}
$ and $\{\psi_n:n\geq0\}$
are orthonormal bases for $M\ominus zM$ and $M\ominus wM$,
respectively.
\end{lemma}

The preceding bases connect the Toeplitz matrices directly with Yang's
numerical invariants \cite{Ya3,Fat}. 

\begin{lemma}
\label{lem:numerical-invariants}
Let $M$ be a homogeneous submodule and let
$\{\phi_n:n\geq0\}$ and $\{\psi_n:n\geq0\}$ be homogeneous
orthonormal bases of $M\ominus zM$ and $M\ominus wM$, respectively.
Then
\begin{equation*}
\Sigma_0(M)
=
\sum_{n=0}^{\infty}
|\langle\phi_n,\psi_n\rangle|^2,
\end{equation*}
and for $j\geq1$,
\begin{equation*}
\label{eq:Sigmaj-basis}
\Sigma_j(M)
=
\sum_{n=0}^{\infty}
\left|
\left\langle
w^j\phi_n,z^j\psi_n
\right\rangle
\right|^2.
\end{equation*}
\end{lemma}

It will be convenient to write
\begin{equation*}
\label{eq:alpha-definition}
\alpha_{n,j}
=
\left\langle
w^j\phi_n,z^j\psi_n
\right\rangle,
\qquad n,j\geq0.
\end{equation*}
Thus
\[
\Sigma_j(M)
=
\sum_{n=0}^{\infty}|\alpha_{n,j}|^2.
\]

For the zeroth invariant, the cofactor representation simplifies
considerably. The following identities are given in
\cite[Corollary~2.2]{Fat}.

\begin{lemma}
\label{lem:basic-inner-products}
Let $M=[p]$ be a homogeneous principal submodule. Then, for every
$n\geq0$,
\[
\langle\phi_n,\psi_n\rangle
=\frac{A^n_{0,n}}{D_n},\,\,\,\,\,\,\,\langle w\phi_n,z\psi_n\rangle
=
-\frac{A^{n+1}_{0,n+1}}{D_{n+1}}.
\]

Consequently, $\langle w\phi_n,z\psi_n\rangle
=-\langle\phi_{n+1},\psi_{n+1}\rangle$.
\end{lemma}

It follows that
\[
\Sigma_0(M)=
\sum_{n=0}^{\infty}
\left|
\frac{A^n_{0,n}}{D_n}
\right|^2,\,\,\,\,\,\,\,\,\,\Sigma_1(M)=
\sum_{n=0}^{\infty}
|\langle w\phi_n,z\psi_n\rangle|^2=
\sum_{n=1}^{\infty}
|\langle\phi_n,\psi_n\rangle|^2.\]

Since
\[
\phi_0=\psi_0=\frac{p}{\|p\|},
\]
we have $|\langle\phi_0,\psi_0\rangle|=1$. Hence, whenever the
corresponding series converge, we obtain $\Sigma_1(M)=\Sigma_0(M)-1$.





\section{Toeplitz matrices, determinants, cofactors, and the core
spectrum of $M_k$}
\label{sec3}
Let
\[
M_k=[(z-w)^k]\subset H^2(\mathbb D^2),\qquad k\geq 1,
\]
and put
\[
p_k(z,w)=(z-w)^k.
\]

We begin by determining the Toeplitz coefficients of $A^n$ explicitly.

Since
\[
p_k(z,w)
=(z-w)^k
=
\sum_{r=0}^{k}
(-1)^{k-r}\binom{k}{r}z^rw^{k-r},
\]
we have
\[
\|p_k\|^2
=
\sum_{r=0}^{k}\binom{k}{r}^2
=
\binom{2k}{k}.
\]
For $m\geq1$, a direct calculation gives
\[
\begin{aligned}
\left\langle p_kw^m,p_kz^m\right\rangle
&=
(-1)^m
\sum_{r=m}^{k}
\binom{k}{r}\binom{k}{r-m}  =
(-1)^m\binom{2k}{k-m},
\qquad 1\leq m\leq k,
\end{aligned}
\]
where the last equality follows from the Vandermonde identity. If $m>k$,
then $\left\langle p_kw^m,p_kz^m\right\rangle=0$.

Consequently, the Toeplitz coefficients associated with $p_k$ are
\[
\gamma_m
=
\begin{cases}
(-1)^m\displaystyle\binom{2k}{k-m},
& 0\leq m\leq k,\\[2mm]
0,
& m>k.
\end{cases}
\]
Since $A^n$ is the Toeplitz Gram matrix introduced in Section~\ref{sec2}, we have
$A^n=(\gamma_{|i-j|})_{i,j=0}^{n}$.

Equivalently,
\begin{equation}
\label{eq:An-entry}
a_{i,j}
=
\begin{cases}
(-1)^{|i-j|}
\displaystyle\binom{2k}{k-|i-j|},
& |i-j|\leq k,\\[2mm]
0,
& |i-j|>k,
\end{cases}
\qquad
0\leq i,j\leq n.
\end{equation}
Hence $A^n$ is a real symmetric $k$-banded Toeplitz matrix of order
$n+1$.


The matrix $A^n$ is a finite Toeplitz section associated with the
pure Fisher--Hartwig symbol $|1-\zeta|^{2k}$. Explicit inversion
formulas for such Toeplitz matrices are known through the
Duduchava--Roch formula; see \cite{BottcherDR}. In particular,
Lemma~\ref{lem:last-column-inverse} can be recovered from the
inverse-entry formula in
\cite[Section~4.2]{GarciaGarciaTierz}, after the corresponding
diagonal sign conjugation.

For completeness, and since only the last column of $(A^n)^{-1}$ is
needed here, we give a short direct finite-difference proof.

\begin{lemma}
\label{lem:last-column-inverse}
For $n\geq0$ and $0\leq j\leq n$,
\begin{equation}
\label{eq:last-column-inverse}
\bigl((A^n)^{-1}\bigr)_{j,n}
=
\frac{
\displaystyle
\binom{k+j}{k}
\binom{k+n-j-1}{k-1}
}{
\displaystyle
\binom{n+2k}{k}
}.
\end{equation}
Since $A^n$ is symmetric, the same formula holds for
$\bigl((A^n)^{-1}\bigr)_{n,j}$.
\end{lemma}

\begin{proof}
Since $A^n$ is the Gram matrix of the linearly independent vectors $p_kz^jw^{n-j}, 0\leq j\leq n$,
it is positive definite and hence invertible.

Define the polynomial
\[
u(x)
=
\frac{(x+1)_k}{k!}
\frac{(n-x+1)_{k-1}}{(k-1)!},
\]
where $(a)_r=a(a+1)\cdots(a+r-1)$ denotes the rising factorial. Then
$\deg u=2k-1$, and for $0\leq j\leq n$,
\[
u(j)
=
\binom{k+j}{k}
\binom{k+n-j-1}{k-1}.
\]
Moreover,
\[
u(-1)=u(-2)=\cdots=u(-k)=0
\]
and
\[
u(n+1)=u(n+2)=\cdots=u(n+k-1)=0.
\]

For $0\leq i\leq n$, consider
\[
S_i
=
\sum_{d=-k}^{k}
(-1)^d
\binom{2k}{k+d}
u(i-d).
\]
Making the change of variables $r=k+d$, we obtain
\[
\begin{aligned}
S_i
&=
(-1)^k
\sum_{r=0}^{2k}
(-1)^r
\binom{2k}{r}
u(i+k-r)=
(-1)^k\Delta^{2k}u(i-k),
\end{aligned}
\]
where $\Delta$ denotes the forward difference operator, i.e.
\[
\Delta^m f(x)
=
\sum_{r=0}^{m}
(-1)^r
\binom{m}{r}
f(x+m-r).
\]

Since
$\deg u=2k-1$, we have
$\Delta^{2k}u\equiv0$. Thus
$S_i=0$.

If $0\leq i\leq n-1$, every term in the above sum for which
$i-d\notin\{0,1,\ldots,n\}$ vanishes by the preceding zero relations.
Hence, using \eqref{eq:An-entry},
\[
\sum_{\ell=0}^{n}a_{i,\ell}u(\ell)=0,
\qquad 0\leq i\leq n-1.
\]

For $i=n$, the same argument applies except for the additional point
$n+k$. Since
\[
u(n+k)
=
(-1)^{k-1}\binom{n+2k}{k},
\]
we obtain
\[
\sum_{\ell=0}^{n}a_{n,\ell}u(\ell)
=
-(-1)^ku(n+k)
=
\binom{n+2k}{k}.
\]
Therefore,
\[
A^n
\begin{pmatrix}
u(0)\\
u(1)\\
\vdots\\
u(n)
\end{pmatrix}
=
\binom{n+2k}{k}
\begin{pmatrix}
0\\
0\\
\vdots\\
1
\end{pmatrix}.
\]
It follows that
\[
\bigl((A^n)^{-1}\bigr)_{j,n}
=
\frac{u(j)}{\binom{n+2k}{k}},
\]
which is exactly \eqref{eq:last-column-inverse}.
\end{proof}

We can now compute the determinants $D_n$.

\begin{theorem}
\label{thm:determinant-z-w-k}
Let $M_k=[(z-w)^k], k\geq1$. Then
\begin{equation*}
\label{eq:Dn-product}
D_n
=
\prod_{r=0}^{n-1}
\frac{r!(2k+r)!}{(k+r)!^2},
\qquad n\geq1.
\end{equation*}
Moreover, the consecutive determinants satisfy
\begin{equation}
\label{eq:eq:Dn-ratio}
\frac{D_{n+1}}{D_n}
=
\frac{n!(n+2k)!}{(n+k)!^2}
=
\frac{\binom{n+2k}{k}}{\binom{n+k}{k}},
\qquad n\geq0.
\end{equation}
In particular,
\[
\det A^n
=
D_{n+1}
=
\prod_{r=0}^{n}
\frac{r!(2k+r)!}{(k+r)!^2}.
\]
\end{theorem}

\begin{proof}
We first verify \eqref{eq:eq:Dn-ratio} for $n=0$. 

By definition, $D_0=1$. Moreover, $D_1=\det A^0$. Since $A^0=(\gamma_0)$
and $\gamma_0=\binom{2k}{k}$,
we have $D_1=\binom{2k}{k}$. Therefore
\[
\frac{D_1}{D_0}
=
\binom{2k}{k}
=
\frac{0!(2k)!}{(k!)^2}
=
\frac{\binom{2k}{k}}{\binom{k}{k}},
\]
so \eqref{eq:eq:Dn-ratio} holds for $n=0$.

Now let $n\geq1$. Since the $(n,n)$-cofactor of $A^n$ is
$A^n_{n,n}=
\det A^{n-1}=D_n$,
the adjugate formula gives
\[
\bigl((A^n)^{-1}\bigr)_{n,n}
=
\frac{D_n}{D_{n+1}}.
\]
On the other hand, Lemma~\ref{lem:last-column-inverse} gives
\[
\bigl((A^n)^{-1}\bigr)_{n,n}
=
\frac{\binom{n+k}{k}}{\binom{n+2k}{k}}.
\]
Consequently,
\[
\frac{D_{n+1}}{D_n}
=
\frac{\binom{n+2k}{k}}{\binom{n+k}{k}}
=
\frac{n!(n+2k)!}{(n+k)!^2},
\qquad n\geq1.
\]
Together with the case $n=0$ considered above, this proves
\eqref{eq:eq:Dn-ratio} for every $n\geq0$.

Finally, since $D_0=1$, iteration of \eqref{eq:eq:Dn-ratio} yields, for
$n\geq1$,
\[
\begin{aligned}
D_n
&=
\prod_{r=0}^{n-1}
\frac{D_{r+1}}{D_r}=
\prod_{r=0}^{n-1}
\frac{r!(2k+r)!}{(k+r)!^2}.
\end{aligned}
\]
This completes the proof.
\end{proof}

\begin{remark}
\label{rem:Dn-first-values}
For later use, we record $D_1=\binom{2k}{k}$,
and $D_{n+1}=
D_n\frac{n!(n+2k)!}{(n+k)!^2}$.
Thus, the determinant formula is compatible with $D_1=\|p_k\|^2$.
\end{remark}

We next determine the algebraic cofactors in the last row of $A^n$.

\begin{theorem}
\label{thm:last-row-cofactor-z-w-k}
Let $M_k=[(z-w)^k]$, where $k\geq1$. For every $n\geq0$ and
$0\leq \ell\leq n$, the algebraic cofactor of the entry in row $n$
and column $\ell$ of $A^n$ is
\begin{equation*}
\label{eq:last-row-cofactor}
A^n_{n,\ell}
=
D_n
\frac{
\displaystyle
\binom{k+\ell}{k}
\binom{k+n-\ell-1}{k-1}
}{
\displaystyle
\binom{k+n}{k}
}.
\end{equation*}
In particular, every algebraic cofactor in the last row is positive.
\end{theorem}

\begin{proof}
By the adjugate identity \eqref{eq:inverse-cofactor},
\[
A^n_{n,\ell}
=
\det(A^n)\bigl((A^n)^{-1}\bigr)_{\ell,n}.
\]
Using $\det A^n=D_{n+1}$ and Lemma~\ref{lem:last-column-inverse},
we obtain
\[
A^n_{n,\ell}
=
D_{n+1}
\frac{
\binom{k+\ell}{k}
\binom{k+n-\ell-1}{k-1}
}{
\binom{n+2k}{k}
}.
\]
By \eqref{eq:eq:Dn-ratio}, we have $\frac{D_{n+1}}{\binom{n+2k}{k}}
=
\frac{D_n}{\binom{n+k}{k}}$.
Therefore, we obtain
\[
A^n_{n,\ell}
=
D_n
\frac{
\binom{k+\ell}{k}
\binom{k+n-\ell-1}{k-1}
}{
\binom{n+k}{k}
},
\]
as required.
\end{proof}

\begin{remark}
\label{rem:first-row-cofactor-z-w-k}
The matrix $A^n$ is centrosymmetric. Hence its inverse, and therefore its
cofactor matrix, is also centrosymmetric. Consequently, $A^n_{0,\ell}
=A^n_{n,n-\ell}$.

Using Theorem~\ref{thm:last-row-cofactor-z-w-k}, we obtain
\begin{equation*}
\label{eq:first-row-cofactor}
A^n_{0,\ell}
=
D_n
\frac{
\displaystyle
\binom{k+n-\ell}{k}
\binom{k+\ell-1}{k-1}
}{
\displaystyle
\binom{k+n}{k}
},
\qquad 0\leq\ell\leq n.
\end{equation*}
In particular, we know that
$A^n_{0,0}=A^n_{n,n}=D_n$
and
\begin{equation}
\label{eq:opposite-corner-cofactors}
A^n_{0,n}
=
A^n_{n,0}
=
\frac{k}{n+k}D_n.
\end{equation}
These formulas will be used both in the spectral analysis below and in
the computation of the numerical invariants in Section~\ref{sec4}.
\end{remark}


In general, determining the spectral data of the core operator
associated with a submodule is difficult. For homogeneous principal
submodules, however, the eigenvalues can be expressed in terms of the
determinants $D_n$ and the corner cofactors $A^n_{0,n}$; see
\cite[Theorem~3.7]{Fat}. Combining this general formula with
Theorem~3.1 and \eqref{eq:opposite-corner-cofactors}, we obtain a particularly
simple spectrum for $M_k$.

\begin{theorem}
\label{thm:core-spectrum}
Let $M_k=[(z-w)^k], k\geq1$, and let $C_{M_k}$ be its core operator. Then the eigenvalues of
$C_{M_k}$ are
\[
\{0,1\}
\cup
\left\{
\pm\frac{k}{n+k}:n\geq1
\right\}.
\]
The eigenvalue $0$ has infinite multiplicity, while the eigenvalue $1$
is simple. Moreover,
\[
\sigma(C_{M_k})
=
\{0,1\}
\cup
\left\{
\pm\frac{k}{n+k}:n\geq1
\right\}.
\]
\end{theorem}

\begin{proof}
For a homogeneous principal submodule, the nontrivial eigenvalues of
the core operator are given by
\[
\pm
\left(
1-
\frac{
\left(
D_n^2-(A^n_{0,n})^2
\right)^2
}{
D_{n-1}D_n^2D_{n+1}
}
\right)^{1/2},
\qquad n\geq1.
\]
In the present setting all the quantities involved are positive and
real. By \eqref{eq:opposite-corner-cofactors}, we have $A^n_{0,n}=
\frac{k}{n+k}D_n$.

Hence
\[
\begin{aligned}
D_n^2-(A^n_{0,n})^2
&=
D_n^2
\left(
1-\frac{k^2}{(n+k)^2}
\right)=
D_n^2
\frac{n(n+2k)}{(n+k)^2}.
\end{aligned}
\]

Using Theorem \ref{thm:determinant-z-w-k}, we obtain
\[
\begin{aligned}
&
\frac{
\left(
D_n^2-(A^n_{0,n})^2
\right)^2
}{
D_{n-1}D_n^2D_{n+1}
}=
\frac{D_n^2}{D_{n-1}D_{n+1}}
\left(
1-\frac{k^2}{(n+k)^2}
\right)^2=
\frac{(n+k)^2}{n(n+2k)}
\left(
\frac{n(n+2k)}{(n+k)^2}
\right)^2=
\frac{n(n+2k)}{(n+k)^2}.
\end{aligned}
\]
It follows that
\[
\begin{aligned}
1-
\frac{
\left(
D_n^2-(A^n_{0,n})^2
\right)^2
}{
D_{n-1}D_n^2D_{n+1}
}
&=
1-
\frac{n(n+2k)}{(n+k)^2}=
\frac{k^2}{(n+k)^2}.
\end{aligned}
\]
Thus the corresponding eigenvalues are
\[
\pm\frac{k}{n+k},
\qquad n\geq1.
\]


The eigenvalue $1$ is simple. Indeed, for a singly generated
homogeneous submodule,
\[
E_1(C_{M_k})
=
M_k\ominus(zM_k+wM_k)
=
\mathbb Cp_k,
\]
where $p_k=(z-w)^k$. Hence, we have 
\[
\dim E_1(C_{M_k})=1.
\]

Moreover, $C_{M_k}$ vanishes on $M_k^\perp$. By the graded
decomposition
\[
M_k=\bigoplus_{n=0}^{\infty}p_k\mathcal H_n,
\]
we have $M_k\cap\mathcal H_{n+k}=p_k\mathcal H_n, n\geq0$. Since multiplication by the nonzero polynomial $p_k$ is injective,
\[
\dim(p_k\mathcal H_n)=n+1,
\qquad
\dim\mathcal H_{n+k}=n+k+1.
\]
Hence $\dim\bigl(\mathcal H_{n+k}\ominus p_k\mathcal H_n\bigr)=k$.
These $k$-dimensional subspaces are contained in $M_k^\perp$ and are
mutually orthogonal as $n$ varies. Thus $M_k^\perp$ is
infinite-dimensional, and consequently $0$ has infinite multiplicity.



Finally, we konw that
\[
\frac{k}{n+k}\longrightarrow0
\qquad(n\to\infty).
\]
Since $C_{M_k}$ is compact and self-adjoint, every nonzero point of its
spectrum is an eigenvalue and the only possible accumulation point of
the spectrum is $0$. Therefore,
\[
\sigma(C_{M_k})
=
\{0,1\}
\cup
\left\{
\pm\frac{k}{n+k}:n\geq1
\right\}.
\]
\end{proof}

\begin{remark}
The spectral behavior of the family $[(z-w)^k]$ differs substantially
from that of $[z^k-w^k]$. For the latter family, the nonzero spectrum
of the core operator is independent of $k$. In contrast,
Theorem~\ref{thm:core-spectrum} shows that for $M_k=[(z-w)^k]$, the largest nontrivial positive eigenvalue is $
\lambda_k=\frac{k}{k+1}$.
Hence, we have
$k=\frac{\lambda_k}{1-\lambda_k}$, so the core spectrum itself determines the parameter $k$. Thus, for
the repeated-factor family considered here, the order of the factor
$z-w$ is already visible in the spectral data of the core operator.
\end{remark}

\section{Numerical invariants and strict monotonicity}
\label{sec4}
For $j\geq1$, define
$\alpha_{n,j}^{(k)}
:=
\langle w^j\phi_n,z^j\psi_n\rangle$.


For $t\in\mathbb Z$, set
\[
\gamma_t
=
\begin{cases}
(-1)^{|t|}
\displaystyle\binom{2k}{k-|t|},
& |t|\leq k,\\[2mm]
0,& |t|>k.
\end{cases}
\]
Then
\[
\left\langle
p_k z^a w^{n-a+j},
p_k z^{b+j}w^{n-b}
\right\rangle
=
\gamma_{j+b-a}.
\]
Consequently, Lemma~\ref{lem:defect-bases} gives
\begin{equation}
\label{eq:alpha-double-sum}
\alpha_{n,j}^{(k)}
=
\frac{1}{D_nD_{n+1}}
\sum_{a,b=0}^{n}
A^n_{0,a}\,
\gamma_{j+b-a}\,
A^n_{n,b}.
\end{equation}

We first reduce \eqref{eq:alpha-double-sum} to a single finite sum.

\begin{lemma}
\label{lem:alpha-unified-sum}
For every $n\geq0$ and $j\geq1$,
\begin{equation}
\label{eq:alpha-unified-sum}
\begin{aligned}
\alpha_{n,j}^{(k)}
={}&
\frac{1}{\binom{n+k}{k}}
\sum_{s=1}^{k}
(-1)^s k\binom{k-1}{s-1}
\frac{(n+1)_{s-1}}{(n+k+1)_s}
\binom{n-j+s+k}{k}
\binom{j-s+k-1}{k-1}.
\end{aligned}
\end{equation}
Here $(a)_m=a(a+1)\cdots(a+m-1)$ denotes the rising factorial. Equivalently, only the indices
$\max\{1,j-n\}
\leq s\leq
\min\{k,j\}$
can contribute to the sum. 
\end{lemma}

\begin{proof}
Throughout this lemma, for integers $a$ and $b$, we use the convention
\[
\binom{a}{b}=0
\qquad\text{unless}\qquad
0\leq b\leq a.
\]

For $\ell\geq1$, define
\[
X_{\ell}^{(n)}
:=
\sum_{a=0}^{n}
A^n_{0,a}\gamma_{\ell-a}.
\]
Then \eqref{eq:alpha-double-sum} can be written as
\begin{equation}
\label{eq:alpha-via-X}
\alpha_{n,j}^{(k)}
=
\frac{1}{D_nD_{n+1}}
\sum_{b=0}^{n}
X_{j+b}^{(n)}A^n_{n,b}.
\end{equation}

If $1\leq\ell\leq n$, then $(\gamma_{\ell-a})_{a=0}^{n}$
is the $\ell$-th row of $A^n$. Hence the cofactor identity gives $X_{\ell}^{(n)}=0,
1\leq\ell\leq n$. Moreover, $X_{\ell}^{(n)}=0,\ell\geq n+k+1$, because $\gamma_{\ell-a}=0$ for every $0\leq a\leq n$.

It therefore remains to determine $X_{n+s}^{(n)},
1\leq s\leq k$.
Let $h=n-a$.
Since $\gamma_{s+h}=0$ whenever $s+h>k$, only the indices $0\le h\le \min\{n,k-s\}$ can contribute. Hence
\[
X_{n+s}^{(n)}
=
\sum_{h=0}^{\min\{n,k-s\}}
A^n_{0,n-h}\gamma_{s+h}.
\]

By the first-row cofactor formula from Section~3,
\[
\frac{A^n_{0,n-h}}{D_n}
=
\frac{
\displaystyle
\binom{k+h}{k}
\binom{n+k-h-1}{k-1}
}{
\displaystyle
\binom{n+k}{k}
},
\qquad 0\le h\le n,
\]
whereas
\[
\gamma_{s+h}
=
(-1)^{s+h}
\binom{2k}{k-s-h}.
\]
Therefore
\[
\frac{X_{n+s}^{(n)}}{D_n}
=
\frac{(-1)^s}{\binom{n+k}{k}}
\sum_{h=0}^{\min\{n,k-s\}}
(-1)^h
\binom{k+h}{k}
\binom{n+k-h-1}{k-1}
\binom{2k}{k-s-h}.
\]

With the zero-binomial convention adopted above, the upper limit may
be extended from $\min\{n,k-s\}$ to $k-s$, since for $h>n$ one has
\[
\binom{n+k-h-1}{k-1}=0.
\]
Thus
\[
\begin{aligned}
\frac{X_{n+s}^{(n)}}{D_n}
={}&
\frac{(-1)^s}{\binom{n+k}{k}}
\sum_{h=0}^{k-s}
(-1)^h
\binom{k+h}{k}
\binom{n+k-h-1}{k-1}
\binom{2k}{k-s-h}.
\end{aligned}
\]

To evaluate the finite sum, we use the Pfaff--Saalsch\"utz summation formula; see, for example, \cite[Section~2.3]{Slater}.
For $N\in\mathbb N_0$,
\[
{}_3F_2
\left(
\begin{matrix}
-N,a,b\\
c,1+a+b-c-N
\end{matrix};1
\right)
=
\frac{(c-a)_N(c-b)_N}
     {(c)_N(c-a-b)_N},
\]
provided that the denominator factors involved are nonzero.


By direct computation, we obtain
\[
\begin{aligned}
&\sum_{h=0}^{k-s}
(-1)^h
\binom{k+h}{k}
\binom{n+k-h-1}{k-1}
\binom{2k}{k-s-h}=
\binom{n+k-1}{k-1}
\binom{2k}{k-s}
{}_3F_2
\left(
\begin{matrix}
k+1,-n,-(k-s)\\
-(n+k-1),k+s+1
\end{matrix};1
\right).
\end{aligned}
\]

Now set $N=k-s, a=k+1,
b=-n, c=k+s+1$. Then $1+a+b-c-N=-(n+k-1)$, so the above ${}_3F_2(1)$ is precisely of
Pfaff--Saalsch\"utz type. Notice also that $0\leq h\leq k-s\leq k-1\leq n+k-1$,
so none of the factors in $(-(n+k-1))_h$ vanishes in the effective summation range. Moreover, $(k+s+1)_{k-s}>0,
(n+s)_{k-s}>0$. Hence all denominator factors occurring in the above
Pfaff--Saalsch\"utz summation are nonzero.


Therefore,
\[
\begin{aligned}
&{}_3F_2
\left(
\begin{matrix}
k+1,-n,-(k-s)\\
-(n+k-1),k+s+1
\end{matrix};1
\right)=
\frac{
(s)_{k-s}(n+k+s+1)_{k-s}
}{
(k+s+1)_{k-s}(n+s)_{k-s}
}.
\end{aligned}
\]
Consequently,
\[
\begin{aligned}
&\sum_{h=0}^{k-s}
(-1)^h
\binom{k+h}{k}
\binom{n+k-h-1}{k-1}
\binom{2k}{k-s-h}=
k\binom{k-1}{s-1}
\binom{n+2k}{k}
\frac{(n+1)_{s-1}}{(n+k+1)_s}.
\end{aligned}
\]
Consequently,
\begin{equation}
\label{eq:X-n-s}
\frac{X_{n+s}^{(n)}}{D_n}
=
(-1)^s
k\binom{k-1}{s-1}
\frac{\binom{n+2k}{k}}{\binom{n+k}{k}}
\frac{(n+1)_{s-1}}{(n+k+1)_s}.
\end{equation}

Returning to \eqref{eq:alpha-via-X}, the only possible nonzero terms
occur when $j+b=n+s,
1\leq s\leq k$. Therefore,
\[
\alpha_{n,j}^{(k)}
=
\frac{1}{D_nD_{n+1}}
\sum_{s=1}^{k}
X_{n+s}^{(n)}
A^n_{n,n-j+s},
\]
where a cofactor with an index outside $\{0,\ldots,n\}$ is understood
to be zero.

By Theorem~\ref{thm:last-row-cofactor-z-w-k},
\[
A^n_{n,n-j+s}
=
D_n
\frac{
\displaystyle
\binom{n-j+s+k}{k}
\binom{j-s+k-1}{k-1}
}{
\displaystyle
\binom{n+k}{k}
}.
\]
Using also
\[
\frac{D_n}{D_{n+1}}
=
\frac{\binom{n+k}{k}}{\binom{n+2k}{k}},
\]
and substituting \eqref{eq:X-n-s}, we obtain
\[
\begin{aligned}
\alpha_{n,j}^{(k)}
={}&
\frac{1}{\binom{n+k}{k}}
\sum_{s=1}^{k}
(-1)^s
k\binom{k-1}{s-1}
\frac{(n+1)_{s-1}}{(n+k+1)_s}
\binom{n-j+s+k}{k}
\binom{j-s+k-1}{k-1}.
\end{aligned}
\]
This proves the desired formula.

Finally,
\[
\binom{n-j+s+k}{k}\neq0
\quad\Longrightarrow\quad
s\geq j-n,
\]
whereas
\[
\binom{j-s+k-1}{k-1}\neq0
\quad\Longrightarrow\quad
s\leq j.
\]
Together with $1\leq s\leq k$, this gives the effective range $\max\{1,j-n\}\leq s\leq\min\{k,j\}$.
\end{proof}


An immediate consequence is the following lower support bound.

\begin{corollary}
\label{cor:lower-support}
For every $j\geq1$,
\[
\alpha_{n,j}^{(k)}=0
\qquad\text{whenever}\qquad
n<j-k.
\]
\end{corollary}

\begin{proof}
If $n<j-k$, then $j-n>k$. Hence there is no integer $s$ satisfying
\[
s\geq j-n
\qquad\text{and}\qquad
1\leq s\leq k.
\]
Thus every term in the finite-sum formula vanishes.
\end{proof}


The following adjacent relation is one of the main structural
identities of the paper.

\begin{theorem}
\label{thm:adjacent-relation}
For every $j\ge 1$ and $n\ge \max\{0,j-k\}$, the inner products $\alpha_{n,j}^{(k)}$ satisfy
\begin{equation}
\label{eq:adjacent-alpha}
(n+k+j+2)\alpha_{n+1,j+1}^{(k)}
=
(n-j+k+1)\alpha_{n,j+1}^{(k)}
+(n+k+j)\alpha_{n,j}^{(k)}
-(n-j+k+1)\alpha_{n+1,j}^{(k)}.
\end{equation}
\end{theorem}

\begin{proof}
We first treat the case $k=1$.  By \eqref{eq:alpha-unified-sum}, only $s=1$
contributes, and
\[
\alpha_{n,j}^{(1)}
=
-\frac{n-j+2}{(n+1)(n+2)}.
\]
Hence, we have 
\[
\alpha_{n+1,j+1}^{(1)}
=
-\frac{n-j+2}{(n+2)(n+3)},\qquad
\alpha_{n,j+1}^{(1)}
=
-\frac{n-j+1}{(n+1)(n+2)},
\]
and
\[
\alpha_{n+1,j}^{(1)}
=
-\frac{n-j+3}{(n+2)(n+3)}.
\]
Substitution into \eqref{eq:adjacent-alpha} shows that both sides are equal to
\[
-\frac{(n-j+2)(n+j+3)}{(n+2)(n+3)}.
\]
Thus \eqref{eq:adjacent-alpha} holds for $k=1$.

We now assume $k\ge2$.  Denote the $s$-th summand in \eqref{eq:alpha-unified-sum}
by $\tau_{n,j}^{(k)}(s)$, so that $\alpha_{n,j}^{(k)}
=\sum_{s=1}^{k}\tau_{n,j}^{(k)}(s)$.

Whenever the corresponding binomial coefficients are nonzero,
\[
\begin{aligned}
\tau_{n,j}^{(k)}(s)
={}&
(-1)^s k
\frac{
(n+s-1)!(n-j+s+k)!(j-s+k-1)!
}{
(s-1)!(k-s)!
}
\\
&\times
\frac{1}{
(n+k+s)!(n-j+s)!(j-s)!
}.
\end{aligned}
\]

Set
\[
\begin{aligned}
\mathcal E_{n,j}^{(k)}(s)
={}&
(n+k+j+2)\tau_{n+1,j+1}^{(k)}(s)
-(n-j+k+1)\tau_{n,j+1}^{(k)}(s)
\\
&-(n+k+j)\tau_{n,j}^{(k)}(s)
+(n-j+k+1)\tau_{n+1,j}^{(k)}(s).
\end{aligned}
\]
We shall show that $\mathcal E_{n,j}^{(k)}(s)$ is a first difference.

For the interior indices, direct division of the factorial expressions gives
\begin{equation}
\label{eq:ratio-1}
\begin{aligned}
\frac{\tau_{n+1,j+1}^{(k)}(s)}
     {\tau_{n,j}^{(k)}(s)}
&=
\frac{(n+s)(j-s+k)}
     {(n+k+s+1)(j-s+1)},\\
\frac{\tau_{n,j+1}^{(k)}(s)}
     {\tau_{n,j}^{(k)}(s)}
&=
\frac{(j-s+k)(n-j+s)}
     {(j-s+1)(n-j+s+k)},\\
\frac{\tau_{n+1,j}^{(k)}(s)}
     {\tau_{n,j}^{(k)}(s)}
&=
\frac{(n+s)(n-j+s+k+1)}
     {(n+k+s+1)(n-j+s+1)}.
\end{aligned}
\end{equation}
Moreover,
\begin{equation}
\label{eq:ratio-s}
\frac{\tau_{n,j}^{(k)}(s+1)}
     {\tau_{n,j}^{(k)}(s)}
=
-
\frac{
(n+s)(n-j+s+k+1)(k-s)(j-s)
}{
s(n+k+s+1)(n-j+s+1)(j-s+k-1)
}.
\end{equation}

Define
\[
R_{n,j}^{(k)}(s)
=
\frac{
(s-1)(s-j-k)(2n+k-2j+2s)
}{
(j-s+1)(n-j+s+k)
}.
\]
Substituting \eqref{eq:ratio-1}--\eqref{eq:ratio-s} into the definition of
$\mathcal E_{n,j}^{(k)}(s)$ and simplifying yields
\begin{equation}
\label{eq:E-ratio}
\frac{\mathcal E_{n,j}^{(k)}(s)}
     {\tau_{n,j}^{(k)}(s)}
=
R_{n,j}^{(k)}(s+1)
\frac{\tau_{n,j}^{(k)}(s+1)}
     {\tau_{n,j}^{(k)}(s)}
-
R_{n,j}^{(k)}(s)
\end{equation}
on the common interior support.

For completeness, if $x=n-j+s, y=j-s$,
then the four summands in $\mathcal E_{n,j}^{(k)}(s)$ have common
support for $x\ge1$, $y\ge0$, while the quotient calculation above
requires $x\ge1, y\ge1,  1\le s\le k-1$.

We now remove these interior restrictions.  Define
\begin{equation*}
\label{eq:telescoping-certificate}
\begin{aligned}
G_{n,j}^{(k)}(s)
={}&
(-1)^{s+1}
\frac{k(s-1)}{(k-1)!}
\binom{k-1}{s-1}
\bigl(2(n+s-j)+k\bigr)
\\
&\times
\frac{
(n-j+s+1)_{k-1}
(j-s+2)_{k-1}
}{
(n+s)_{k+1}
},
\qquad 1\le s\le k+1.
\end{aligned}
\end{equation*}

On the common interior support, direct cancellation gives $G_{n,j}^{(k)}(s)
=\tau_{n,j}^{(k)}(s)R_{n,j}^{(k)}(s)$. Hence \eqref{eq:E-ratio} implies
$\mathcal E_{n,j}^{(k)}(s)
=G_{n,j}^{(k)}(s+1)-G_{n,j}^{(k)}(s)$ there.

It remains only to justify the boundary indices.  Put
\[
x=n-j+s,\qquad
y=j-s,\qquad
F=n+s,\qquad
J=n+k+s+1,
\]
and
\[
\eta_s
=
(-1)^s
\frac{k}{(k-1)!}
\binom{k-1}{s-1}
\frac1{(n+s)_{k+1}}.
\]
Then the binomial/rising-factorial form of \eqref{eq:alpha-unified-sum} gives $\tau_{n,j}^{(k)}(s)
=\eta_s (x+1)_k(y+1)_{k-1}$,
together with
\[
\begin{aligned}
\tau_{n+1,j+1}^{(k)}(s)
&=
\eta_s\frac{F}{J}(x+1)_k(y+2)_{k-1},\,\,\,\,\,\,\,\,\,\,
\tau_{n,j+1}^{(k)}(s)
&=
\eta_s(x)_k(y+2)_{k-1},\\
\tau_{n+1,j}^{(k)}(s)
&=
\eta_s\frac{F}{J}(x+2)_k(y+1)_{k-1}.
\end{aligned}
\]
Also, we have
\[
G_{n,j}^{(k)}(s)
=
-\eta_s(s-1)(2x+k)
(x+1)_{k-1}(y+2)_{k-1},
\]
and
\[
G_{n,j}^{(k)}(s+1)
=
\eta_s\frac{F}{J}(k-s)(2x+k+2)
(x+2)_{k-1}(y+1)_{k-1}.
\]

Let $U=(x+2)_{k-2},
V=(y+2)_{k-2}$.
Using the preceding finite-product identities, both sides can be written
without division as
\[
\mathcal E_{n,j}^{(k)}(s)=\eta_sUV\,\mathcal P,
\qquad
G_{n,j}^{(k)}(s+1)-G_{n,j}^{(k)}(s)
=\eta_sUV\,\mathcal Q.
\]
Thus it is enough to verify the algebraic identity $\mathcal P=\mathcal Q$, where 
\[
\begin{aligned}
\mathcal P
={}&
\frac{F}{J}(F+y+k+2)(x+1)(x+k)(y+k)-(x-s+k+1)x(x+1)(y+k)
\\
&-(F+y+k)(x+1)(x+k)(y+1)
+\frac{F}{J}(x-s+k+1)(x+k)(x+k+1)(y+1),
\end{aligned}
\]
and
\[
\begin{aligned}
\mathcal Q
={}&
\frac{F}{J}(k-s)(2x+k+2)(x+k)(y+1)+(s-1)(2x+k)(x+1)(y+k).
\end{aligned}
\]
Notice that no division by $U$ or $V$ is involved, which is essential at
the boundary where either factor may vanish. 

A direct grouping gives
\[
\mathcal P-\mathcal Q
=
(x+1)(y+k)\mathcal A
+
(x+k)(y+1)\mathcal B,
\]
where
\[
\mathcal A
=
\frac{(x+k)(y+1)(2F+k+1)}{J},
\qquad
\mathcal B
=
-\frac{(y+k)(x+1)(2F+k+1)}{J}.
\]
Consequently, $\mathcal P-\mathcal Q=0$.

Since $n\geq \max\{0,j-k\}, j\geq1, 1\leq s\leq k$, we have $x=n-j+s\geq1-k,
y=j-s\geq1-k$. Thus all admissible boundary indices lie in the finite ranges where
the vanishing rising-factorial factors agree with the zero-binomial
convention used above. Moreover,
\[
J=n+k+s+1>0,
\qquad
(n+s)_{k+1}>0.
\]
Since the identity $\mathcal P=\mathcal Q$ was obtained without
dividing by any of $x,\ y,\ x+1,\ y+1,\ k-s,\ U,\ V$,
it remains valid at all admissible boundary indices, including the
terminal case $s=k$.


We have therefore proved, for every $1\le s\le k$,
\begin{equation*}
\label{eq:telescoping-identity}
\mathcal E_{n,j}^{(k)}(s)
=
G_{n,j}^{(k)}(s+1)-G_{n,j}^{(k)}(s).
\end{equation*}

Finally, $G_{n,j}^{(k)}(1)=0$ because of the factor $s-1$, and $G_{n,j}^{(k)}(k+1)=0$ because $\binom{k-1}{k}=0$.

Therefore,
\[
\begin{aligned}
\sum_{s=1}^{k}\mathcal E_{n,j}^{(k)}(s)
&=
\sum_{s=1}^{k}
\left(
G_{n,j}^{(k)}(s+1)-G_{n,j}^{(k)}(s)
\right)=
G_{n,j}^{(k)}(k+1)-G_{n,j}^{(k)}(1)
=0.
\end{aligned}
\]
Since $\alpha_{n,j}^{(k)}
=\sum_{s=1}^{k}\tau_{n,j}^{(k)}(s)$, the definition of $\mathcal E_{n,j}^{(k)}(s)$ now gives
\[
\begin{aligned}
0
={}&
(n+k+j+2)\alpha_{n+1,j+1}^{(k)}
-(n-j+k+1)\alpha_{n,j+1}^{(k)}
-(n+k+j)\alpha_{n,j}^{(k)}
+(n-j+k+1)\alpha_{n+1,j}^{(k)},
\end{aligned}
\]
which is equivalent to \eqref{eq:adjacent-alpha}.
\end{proof}

We next record the value of the first coefficient in each column.

\begin{lemma}
\label{lem:alpha-zero-j}
For $1\leq j\leq k$,
\begin{equation*}
\label{eq:alpha-zero-j}
\alpha_{0,j}^{(k)}
=
(-1)^j
\frac{(k!)^2}{
(k-j)!(k+j)!
}.
\end{equation*}
Consequently, for $1\leq j<k$,
\begin{equation}
\label{eq:alpha-zero-ratio}
\frac{
\alpha_{0,j+1}^{(k)}
}{
\alpha_{0,j}^{(k)}
}
=
-\frac{k-j}{k+j+1}.
\end{equation}
\end{lemma}

\begin{proof}
Put $n=0$ in \eqref{eq:alpha-unified-sum}. The factor
$\binom{-j+s+k}{k}$
can be nonzero only if $s\geq j$, whereas $\binom{j-s+k-1}{k-1}$ can be nonzero only if $s\leq j$. Hence only $s=j$ contributes.
Substituting $s=j$ into \eqref{eq:alpha-unified-sum} gives
\[
\alpha_{0,j}^{(k)}
=
(-1)^j
\frac{(k!)^2}{
(k-j)!(k+j)!
}.
\]
Taking the quotient of two consecutive values yields
\eqref{eq:alpha-zero-ratio}.
\end{proof}

\begin{remark}
For $k=2$, the formulas above recover the previously known numerical
invariants after taking absolute squares. We note that the individual
inner products $\alpha_{n,j}^{(k)}$ may depend on the phase
normalization of the homogeneous defect-space bases, whereas
$\Sigma_j(M_k)$ does not.
\end{remark}

Combining Corollary~\ref{cor:lower-support} with the boundary values in Lemma~\ref{lem:alpha-zero-j} gives
the exact lower support.

\begin{corollary}
\label{cor:exact-support}
For every $j\geq1$, we have $\min\{n\geq0:\alpha_{n,j}^{(k)}\neq0\}=
\max\{0,j-k\}$. More precisely, if $1\leq j\leq k$, then
\[
\alpha_{0,j}^{(k)}
=
(-1)^j
\frac{(k!)^2}{(k-j)!(k+j)!}
\neq0,
\]
whereas if $j > k$, then
\[
\alpha_{j-k,j}^{(k)}
=
(-1)^k k^2
\frac{(j-k+1)_{k-1}^2}
{(j-k+1)_{2k}}
\neq0.
\]
\end{corollary}

\begin{proof}
The case $1\leq j\leq k$ follows from
Lemma~\ref{lem:alpha-zero-j}. Now let $j > k$ and set $n=j-k$. In the finite-sum formula for $\alpha_{n,j}^{(k)}$, the condition $s\geq j-n=k$
shows that only the term $s=k$ can be nonzero. Substitution gives
\[
\alpha_{j-k,j}^{(k)}
=
(-1)^k k^2
\frac{(j-k+1)_{k-1}^2}
{(j-k+1)_{2k}},
\]
which is nonzero. Together with
Corollary~\ref{cor:lower-support}, the result follows.
\end{proof}

The following finite-section contraction theorem is the main
structural result of this section.

\begin{theorem}
\label{thm:finite-section}
For $k\geq1$, $j\geq1$, and $N\geq\max\{0,j-k\}$,
we have
\begin{equation}
\label{eq:finite-section-monotonicity}
\sum_{n=0}^{N}
\left|\alpha_{n,j}^{(k)}\right|^2
>
\sum_{n=0}^{N}
\left|\alpha_{n,j+1}^{(k)}\right|^2.
\end{equation}
\end{theorem}

\begin{proof}
We divide the proof into two cases.

\medskip
\noindent
\textbf{Case 1: $j\geq k$.}

Put $q=j-k\geq0$.
By Corollary~\ref{cor:lower-support}, we have
\[
\alpha_{n,j}^{(k)}=0\quad(n<q),
\qquad
\alpha_{n,j+1}^{(k)}=0\quad(n<q+1).
\]
For $m\geq1$, define
\[
x_m
=
\alpha_{q+m-1,j}^{(k)},
\qquad
y_m
=
\alpha_{q+m,j+1}^{(k)},
\]
and put $y_0=0$.

Taking $n=q+m-1$
in \eqref{eq:adjacent-alpha}, we have
$n-j+k+1=m$,
and
$n+k+j=2j+m-1$.

Hence
\begin{equation}
\label{eq:tail-recurrence}
(2j+m+1)y_m
=
my_{m-1}
+(2j+m-1)x_m
-mx_{m+1}.
\end{equation}

Fix $N\geq q$ and put $M=N-q$. If $M=0$, then we have
\[
\sum_{n=0}^{N}|\alpha_{n,j+1}^{(k)}|^2=0,\,\,\,\,\,\,\,\,\,\,\,\sum_{n=0}^{N}|\alpha_{n,j}^{(k)}|^2
=|\alpha_{q,j}^{(k)}|^2>0,
\]
so the result is immediate.

Assume now $M\geq1$, and set
\[
x=
\begin{pmatrix}
x_1\\
x_2\\
\vdots\\
x_{M+1}
\end{pmatrix},
\qquad
y=
\begin{pmatrix}
y_1\\
y_2\\
\vdots\\
y_M
\end{pmatrix}.
\]
The recurrence \eqref{eq:tail-recurrence} can be written as
\begin{equation*}
\label{eq:tail-matrix-equation}
\mathcal L_M(j)y
=
\mathcal R_M(j)x,
\end{equation*}
where
\[
\mathcal L_M(j)
=
\begin{pmatrix}
2j+2 &        &        &        \\
-2   & 2j+3  &        &        \\
     & -3    & 2j+4  &        \\
     &       & \ddots & \ddots \\
     &       &        & -M & 2j+M+1
\end{pmatrix}_{M\times M},
\]
and
\[
\mathcal R_M(j)
=
\begin{pmatrix}
2j   & -1   &      &      &      \\
     & 2j+1& -2   &      &      \\
     &      & 2j+2& -3   &      \\
     &      &      &\ddots&\ddots\\
     &      &      &      &2j+M-1&-M
\end{pmatrix}_{M\times(M+1)}.
\]
Since $\mathcal L_M(j)$ is lower bidiagonal with strictly positive diagonal entries, it is invertible. Define
\[
T_M(j)
=
\mathcal L_M(j)^{-1}\mathcal R_M(j).
\]
Then
$y=T_M(j)x$.

We claim that $T_M(j)$ is a strict contraction. Consider
\[
\mathcal D_M(j)
=
\mathcal L_M(j)\mathcal L_M(j)^*
-
\mathcal R_M(j)\mathcal R_M(j)^*.
\]
A direct multiplication gives
$(\mathcal D_M(j))_{1,1}=
8j+3$, and, for $2\leq m\leq M$, $(\mathcal D_M(j))_{m,m}=
4(2j+m)$.

The only nonzero off-diagonal entries are
\[
(\mathcal D_M(j))_{m,m+1}
=
(\mathcal D_M(j))_{m+1,m}
=-(2j+2m+1),
\qquad
1\leq m<M.
\]

If $M\geq2$, the strict diagonal-dominance margin in the first row is
\[
(8j+3)-(2j+3)=6j>0.
\]
For every interior row $2\leq m\leq M-1$,
\[
\begin{aligned}
4(2j+m)
&-(2j+2m-1)
-(2j+2m+1)=4j>0.
\end{aligned}
\]
For the last row,
\[
\begin{aligned}
4(2j+M)-(2j+2M-1)
&=
6j+2M+1>0.
\end{aligned}
\]
Hence every row is strictly diagonally dominant. When $M=1$, the matrix
$\mathcal D_M$ consists only of the positive entry $8j+3$.
Hence $\mathcal D_M(j)$ is a real symmetric strictly diagonally dominant matrix
with positive diagonal entries. By Gershgorin's theorem, all eigenvalues of \(\mathcal D_M(j)\) are strictly positive. Therefore, we obtain that \(\mathcal D_M(j) \succ 0\). Here, \(\mathcal D_M(j) \succ 0\) means that \(\mathcal D_M(j)\) is positive definite.

Now we have
\[
\begin{aligned}
I_M-T_M(j)T_M(j)^*
&=
\mathcal L_M(j)^{-1}
\mathcal D_M(j)
\mathcal L_M(j)^{-*} \succ 0.
\end{aligned}
\]
Thus $\|T_M(j)\|<1$.
Moreover, Corollary~\ref{cor:exact-support} gives
\[
x_1
=
\alpha_{q,j}^{(k)}
=
(-1)^k k^2
\frac{(q+1)_{k-1}^2}{(q+1)_{2k}}
\neq0.
\]
Hence $x\neq0$. Since $y=T_M(j)x$
and $\|T_M(j)\|<1$, we obtain
$\|y\|^2<\|x\|^2$.

Using the support conditions, we have
\[
\|x\|^2
=
\sum_{n=0}^{N}|\alpha_{n,j}|^2,\,\,\,\,\,\,\,\,\,\,\,\,\,\,\|y\|^2
=
\sum_{n=0}^{N}|\alpha_{n,j+1}|^2.
\]

This proves \eqref{eq:finite-section-monotonicity} when $j\geq k$.

\medskip
\noindent
\textbf{Case 2: $1\leq j<k$.}

Set $x_n=\alpha_{n,j}^{(k)},
y_n=\alpha_{n,j+1}^{(k)}$. By Lemma~\ref{lem:alpha-zero-j}, we have
\[
y_0
=
-\beta_{k,j}x_0,
\qquad
\beta_{k,j}
=
\frac{k-j}{k+j+1}.
\]
Since $1\leq j<k$, then
$0<\beta_{k,j}<1$.

For $n\geq0$, the adjacent relation
\eqref{eq:adjacent-alpha} becomes
\begin{equation}
\label{eq:front-recurrence}
\begin{aligned}
&(n+k+j+2)y_{n+1}
-(n-j+k+1)y_n=
(n+k+j)x_n
-(n-j+k+1)x_{n+1}.
\end{aligned}
\end{equation}

If $N=0$, then $|y_0|
=\beta_{k,j}|x_0|<
|x_0|$,
and the assertion follows.

Assume $N\geq1$. Define
\[
x=
\begin{pmatrix}
x_0\\
x_1\\
\vdots\\
x_N
\end{pmatrix},
\qquad
y=
\begin{pmatrix}
y_0\\
y_1\\
\vdots\\
y_N
\end{pmatrix}.
\]
The boundary relation and \eqref{eq:front-recurrence} can be written as
$\widehat{\mathcal L}_N(k,j)y
=
\widehat{\mathcal R}_N(k,j)x$,
where
\[
\widehat{\mathcal L}_N(k,j)
=
\begin{pmatrix}
1 & 0 & 0 & \cdots & 0\\
-(k-j+1)
& k+j+2
& 0
& \cdots
& 0\\
0
& -(k-j+2)
& k+j+3
& \ddots
& \vdots\\
\vdots
& \ddots
& \ddots
& \ddots
& 0\\
0
& \cdots
& 0
& -(k-j+N)
& k+j+N+1
\end{pmatrix},
\]
and
\[
\widehat{\mathcal R}_N(k,j)
=
\begin{pmatrix}
-\beta_{k,j} & 0 & 0 & \cdots & 0\\
k+j
& -(k-j+1)
& 0
& \cdots
& 0\\
0
& k+j+1
& -(k-j+2)
& \ddots
& \vdots\\
\vdots
& \ddots
& \ddots
& \ddots
& 0\\
0
& \cdots
& 0
& k+j+N-1
& -(k-j+N)
\end{pmatrix}.
\]

Let
\[
\widehat{\mathcal D}_N(k,j)
=
\widehat{\mathcal L}_N(k,j)
\widehat{\mathcal L}_N(k,j)^*
-
\widehat{\mathcal R}_N(k,j)
\widehat{\mathcal R}_N(k,j)^*.
\]
A direct multiplication gives
\begin{equation*}
\label{eq:D-front-00}
(\widehat{\mathcal D}_N(k,j))_{0,0}
=
1-\beta_{k,j}^2
=
\frac{(2j+1)(2k+1)}
{(k+j+1)^2}>0,
\end{equation*}
and
\begin{equation*}
\label{eq:D-front-01}
(\widehat{\mathcal D}_N(k,j))_{0,1}
=
(\widehat{\mathcal D}_N(k,j))_{1,0}
=
-\frac{2k+1}{k+j+1}.
\end{equation*}
For $1\leq r\leq N$,
$(\widehat{\mathcal D}_N(k,j))_{r,r}
=
4(k+j+r)$, and, for $1\leq r<N$, $(\widehat{\mathcal D}_N(k,j))_{r,r+1}
=(\widehat{\mathcal D}_N(k,j))_{r+1,r}
=-(2k+2r+1)$. All other entries vanish.

The first row is not necessarily diagonally dominant, so we use a
positive diagonal scaling. Define
\[
v=
\begin{pmatrix}
v_0\\
v_1\\
\vdots\\
v_N
\end{pmatrix}=
\begin{pmatrix}
\dfrac{2(k+j+1)}{2j+1}\\[2mm]
1\\
\vdots\\
1
\end{pmatrix}.
\]
Then
\[
(\widehat{\mathcal D}_N(k,j)v)_0
=
\frac{2k+1}{k+j+1}>0.
\]
If $N\geq2$, then
\[
\begin{aligned}
(2j+1)(\widehat{\mathcal D}_N(k,j)v)_1
={}&
2k(2j-1)
+8j^2+6j-1>0,
\end{aligned}
\]
because $j\geq1$. 

For every interior row $2\leq r<N$,
\[
\begin{aligned}
(\widehat{\mathcal D}_N(k,j)v)_r
&=
4(k+j+r)
-(2k+2r-1)
-(2k+2r+1)=4j>0.
\end{aligned}
\]
For $N\geq2$, the last row satisfies
\[
\begin{aligned}
(\widehat{\mathcal D}_N(k,j)v)_N=
4(k+j+N)-(2k+2N-1)=2k+4j+2N+1>0.
\end{aligned}
\]
If $N=1$, then
\[
(\widehat{\mathcal D}_N(k,j)v)_1
=
4(k+j+1)
-
\frac{2(2k+1)}{2j+1}
>0.
\]

Therefore, all components of the vector $\widehat{\mathcal D}_N(k,j) v$ are strictly positive.


Let $V=\operatorname{diag}(v_0,\ldots,v_N)$.
Since $\widehat{\mathcal D}_N(k,j)$ is real symmetric and its off-diagonal entries
are nonpositive, for every $0\leq r\leq N$ we have
\[
(V\widehat{\mathcal D}_N(k,j)V)_{r,r}
-
\sum_{s\neq r}
\left|(V\widehat{\mathcal D}_N(k,j)V)_{r,s}\right|
=
v_r(\widehat{\mathcal D}_N(k,j)v)_r.
\]
Hence $\widehat{\mathcal D}_N(k,j)v>0$ implies that
$V\widehat{\mathcal D}_N(k,j)V$ is strictly diagonally dominant with positive
diagonal entries. Therefore
$\widehat{\mathcal D}_N(k,j)V \succ 0$, and consequently
$\widehat{\mathcal D}_N(k,j) \succ 0$.


Since $\widehat L_N(k,j)$ is lower triangular with strictly positive
diagonal entries, it is invertible. Hence we may define
\[
\widehat T_N(k,j)
=
\widehat{\mathcal L}_N(k,j)^{-1}\widehat {\mathcal R}_N(k,j).
\]
Then
\[
I-\widehat T_N(k,j)\widehat T_N(k,j)^*
=
\widehat{\mathcal L}_N(k,j)^{-1}
\widehat{\mathcal D}_N(k,j)
\widehat{\mathcal L}_N(k,j)^{-*}
\succ 0,
\]
so
\[
\|\widehat T_N(k,j)\|<1.
\]
Since $y=\widehat T_N(k,j)x$
and $x\neq0$, we obtain
\[
\sum_{n=0}^{N}|y_n|^2
<
\sum_{n=0}^{N}|x_n|^2.
\]

This proves
\eqref{eq:finite-section-monotonicity} for $1\leq j<k$ and completes
the proof.
\end{proof}

\begin{remark}
\label{rem:not-termwise}
Theorem~\ref{thm:finite-section} is a finite-section norm inequality. It does not assert that $|\alpha_{n,j}|
\geq
|\alpha_{n,j+1}|$
for every fixed $n$. In fact, such a termwise inequality is false in
general. The monotonicity is produced by the contraction of the entire
finite coefficient vector.
\end{remark}

To pass from finite sections to the infinite numerical invariants while
preserving strictness, we need the large-$n$ behavior of
$\alpha_{n,j}^{(k)}$.

\begin{lemma}
\label{lem:alpha-asymptotic}
For every fixed $k\geq1$ and $j\geq1$,
\begin{equation}
\label{eq:alpha-asymptotic}
\alpha_{n,j}^{(k)}
=
-\frac{k}{n+k}
+
\frac{kj}{(n+k)^2}
+
O(n^{-3}),
\qquad n\to\infty.
\end{equation}
Consequently,
\begin{equation}
\label{eq:alpha-square-difference}
|\alpha_{n,j}^{(k)}|^2
-
|\alpha_{n,j+1}^{(k)}|^2
=
\frac{2k^2}{(n+k)^3}
+
O(n^{-4}).
\end{equation}
\end{lemma}

\begin{proof}
We first treat the case $k=1$ separately. By
Lemma~\ref{lem:alpha-unified-sum}, the sum contains only the term
$s=1$, and hence
\[
\alpha_{n,j}^{(1)}
=
-\frac{n-j+2}{(n+1)(n+2)}.
\]

Put $M=n+1$.
Then
\[
\alpha_{n,j}^{(1)}
=
-\frac{M-j+1}{M(M+1)}
=
-\frac{1}{M}
+
\frac{j}{M(M+1)}.
\]
Since
\[
\frac{1}{M(M+1)}
=
\frac{1}{M^2}
+
O(M^{-3}),
\]
we obtain
\[
\alpha_{n,j}^{(1)}
=
-\frac{1}{M}
+
\frac{j}{M^2}
+
O(M^{-3}).
\]
Since $M=n+1=n+k$ when $k=1$, this is precisely
\[
\alpha_{n,j}^{(1)}
=
-\frac{1}{n+1}
+
\frac{j}{(n+1)^2}
+
O(n^{-3}),
\]
which is the first asserted asymptotic formula for $k=1$.

Moreover,
\[
\begin{aligned}
\left|\alpha_{n,j}^{(1)}\right|^2
-
\left|\alpha_{n,j+1}^{(1)}\right|^2
&=
\frac{(M-j+1)^2-(M-j)^2}
     {M^2(M+1)^2}=
\frac{2M-2j+1}{M^2(M+1)^2}.
\end{aligned}
\]
Therefore
\[
\left|\alpha_{n,j}^{(1)}\right|^2
-
\left|\alpha_{n,j+1}^{(1)}\right|^2
=
\frac{2}{M^3}
+
O(M^{-4})
=
\frac{2}{(n+1)^3}
+
O(n^{-4}),
\]
which is the second asserted asymptotic formula for $k=1$.

For the remainder of the proof, assume $k\geq2$.

Put $M=n+k$.
Rewrite \eqref{eq:alpha-unified-sum} as
\[
\alpha_{n,j}^{(k)}
=
\sum_{s=1}^{k} C_s F_s(M),
\]
where
\[
C_s
=
(-1)^s k
\binom{k-1}{s-1}
\binom{j-s+k-1}{k-1},
\]
and
\[
F_s(M)
=
\frac{
\binom{M-j+s}{k}
}{
\binom{M}{k}
}
\frac{
(M-k+1)_{s-1}
}{
(M+1)_s
}.
\]

For fixed $k,j,s$, we first have
\[
\begin{aligned}
\frac{
\binom{M-j+s}{k}
}{
\binom{M}{k}
}
&=
\prod_{r=0}^{k-1}
\frac{M-j+s-r}{M-r}=
1+\frac{k(s-j)}{M}
+O(M^{-2}).
\end{aligned}
\]
Moreover,
\[
\begin{aligned}
\frac{
(M-k+1)_{s-1}
}{
(M+1)_s
}
&=
\frac1M
\frac{
\displaystyle
\prod_{r=1}^{s-1}
\left(1+\frac{-k+r}{M}\right)
}{
\displaystyle
\prod_{r=1}^{s}
\left(1+\frac{r}{M}\right)
}=
\frac1M
\left[
1+
\frac{k-(k+1)s}{M}
+O(M^{-2})
\right].
\end{aligned}
\]
Therefore,
\begin{equation}
\label{eq:F-s-asymptotic}
F_s(M)
=
\frac1M
+
\frac{k(1-j)-s}{M^2}
+
O(M^{-3}).
\end{equation}

Since $k$ and $j$ are fixed and $1\leq s\leq k$,
the index $s$ ranges over a finite set. Hence the remainder term in \eqref{eq:F-s-asymptotic} may be chosen uniformly in $s$. More
precisely, there exists a constant $C=C(k,j)>0$ such that
\[
\left|
F_s(M)
-
\frac1M
\left(
1+\frac{k(1-j)-s}{M}
\right)
\right|
\leq
\frac{C}{M^3},
\qquad
1\leq s\leq k,
\]
for all sufficiently large $M$.

It remains to evaluate two finite sums involving $C_s$. Put $m=k-1, x=j+k-2$. With $r=s-1$, the standard backward-difference identity
$\nabla^m\binom{x}{m}=1$
gives
\[
\sum_{r=0}^{m}
(-1)^r
\binom{m}{r}
\binom{x-r}{m}
=1.
\]
Hence
\begin{equation}
\label{eq:C-s-sum}
\sum_{s=1}^{k}C_s=-k.
\end{equation}

Next, using
$r\binom{m}{r}
=
m\binom{m-1}{r-1}$, we obtain
\[
\begin{aligned}
\sum_{r=0}^{m}
r(-1)^r
\binom{m}{r}
\binom{x-r}{m}
&=
-m
\sum_{t=0}^{m-1}
(-1)^t
\binom{m-1}{t}
\binom{x-1-t}{m}=
-m(x-m).
\end{aligned}
\]
Since $x-m=j-1$, it follows that
\[
\sum_{r=0}^{m}
(r+1)(-1)^r
\binom{m}{r}
\binom{x-r}{m}
=
k-(k-1)j.
\]
Consequently,
\begin{equation}
\label{eq:sC-s-sum}
\sum_{s=1}^{k}sC_s
=
k\bigl[(k-1)j-k\bigr].
\end{equation}

Using the uniform expansion \eqref{eq:F-s-asymptotic}, we have
\[
\begin{aligned}
\alpha_{n,j}^{(k)}
&=
\sum_{s=1}^{k}C_sF_s(M)=
\sum_{s=1}^{k}
C_s
\left[
\frac1M
+
\frac{k(1-j)-s}{M^2}
+
O(M^{-3})
\right].
\end{aligned}
\]
Since the sum contains only finitely many terms and the remainder in
\eqref{eq:F-s-asymptotic} is uniform in $s$, we have
\[
\sum_{s=1}^{k}C_s\,O(M^{-3})
=
O(M^{-3}).
\]
Therefore,
\[
\begin{aligned}
\alpha_{n,j}^{(k)}
&=
\frac1M
\sum_{s=1}^{k}C_s
+
\frac1{M^2}
\sum_{s=1}^{k}
C_s\bigl[k(1-j)-s\bigr]
+
O(M^{-3}).
\end{aligned}
\]
Using \eqref{eq:C-s-sum} and
\eqref{eq:sC-s-sum}, we have
\[
\begin{aligned}
\sum_{s=1}^{k}
C_s\bigl[k(1-j)-s\bigr]
&=
k(1-j)\sum_{s=1}^{k}C_s
-
\sum_{s=1}^{k}sC_s=
-k^2(1-j)
-
k\bigl[(k-1)j-k\bigr]=
kj.
\end{aligned}
\]
Consequently,
\[
\alpha_{n,j}^{(k)}
=
-\frac{k}{M}
+
\frac{kj}{M^2}
+
O(M^{-3}).
\]
Since $M=n+k$, we obtain
\[
\alpha_{n,j}^{(k)}
=
-\frac{k}{n+k}
+
\frac{kj}{(n+k)^2}
+
O(n^{-3}),
\qquad n\to\infty,
\]
which proves \eqref{eq:alpha-asymptotic}.


Squaring gives
\[
|\alpha_{n,j}^{(k)}|^2
=
\frac{k^2}{(n+k)^2}
-
\frac{2k^2j}{(n+k)^3}
+
O(n^{-4}),
\]
and replacing $j$ by $j+1$ and subtracting yields
\eqref{eq:alpha-square-difference}.
\end{proof}

We can now determine the first two invariants explicitly and prove
strict monotonicity of the complete sequence.

\begin{theorem}
\label{thm:strict-monotonicity}
Let $M_k=[(z-w)^k], k\geq1$. Then
\begin{equation}
\label{eq:Sigma0}
\Sigma_0(M_k)
=
k^2
\sum_{m=k}^{\infty}\frac1{m^2}
=
k^2
\left(
\frac{\pi^2}{6}
-
\sum_{m=1}^{k-1}\frac1{m^2}
\right),
\end{equation}
and
\begin{equation*}
\label{eq:Sigma1}
\Sigma_1(M_k)
=
k^2
\sum_{m=k+1}^{\infty}\frac1{m^2}
=
\Sigma_0(M_k)-1.
\end{equation*}
Moreover,
\begin{equation}
\label{eq:strict-monotonicity}
\Sigma_0(M_k)
>
\Sigma_1(M_k)
>
\Sigma_2(M_k)
>
\Sigma_3(M_k)
>\cdots .
\end{equation}
In particular, Yang's monotonicity conjecture holds for the entire
family $\{[(z-w)^k]:k\geq1\}$,
and the sequence is in fact strictly decreasing.
\end{theorem}

\begin{proof}
By the general formula for $\Sigma_0$ and \eqref{eq:opposite-corner-cofactors}, we have $A^n_{0,n}=
\frac{k}{n+k}D_n$.
Therefore
\[
\begin{aligned}
\Sigma_0(M_k)
&=
\sum_{n=0}^{\infty}
\left|
\frac{A^n_{0,n}}{D_n}
\right|^2=
\sum_{n=0}^{\infty}
\frac{k^2}{(n+k)^2}=
k^2
\sum_{m=k}^{\infty}\frac1{m^2},
\end{aligned}
\]
which proves \eqref{eq:Sigma0}.

For $j=1$, only $s=1$ contributes in \eqref{eq:alpha-unified-sum}.
Hence, we have $\alpha_{n,1}^{(k)}=
-\frac{k}{n+k+1}$.
Therefore,
\[
\Sigma_1(M_k)
=
k^2\sum_{m=k+1}^{\infty}\frac1{m^2}
=
\Sigma_0(M_k)-1.
\]

Consequently, we have $\Sigma_0(M_k)>\Sigma_1(M_k)$.


Now fix $j\geq1$. By
Lemma~\ref{lem:alpha-asymptotic}, we have $|\alpha_{n,j}^{(k)}|^2
=O(n^{-2})$, so both $\Sigma_j(M_k)$ and $\Sigma_{j+1}(M_k)$ converge.

Define the finite-section difference
\[
\Delta_{j,N}^{(k)}
=
\sum_{n=0}^{N}
\left(
|\alpha_{n,j}^{(k)}|^2
-
|\alpha_{n,j+1}^{(k)}|^2
\right).
\]
By Theorem~\ref{thm:finite-section}, $\Delta_{j,N}^{(k)}>0$
whenever $N\geq\max\{0,j-k\}$.

On the other hand,
\eqref{eq:alpha-square-difference} implies that there exists
$N_0=N_0(k,j)$ such that
$|\alpha_{n,j}^{(k)}|^2
-|\alpha_{n,j+1}^{(k)}|^2
>0$
for every $n\geq N_0$.
Increasing $N_0$ if necessary, assume also that
$N_0\geq\max\{0,j-k\}$.
Then $\Delta_{j,N_0}^{(k)}>0$,
and, for every $N\geq N_0$, we have
\[
\begin{aligned}
\Delta_{j,N+1}^{(k)}
-
\Delta_{j,N}^{(k)}
&=
|\alpha_{N+1,j}^{(k)}|^2
-
|\alpha_{N+1,j+1}^{(k)}|^2>0.
\end{aligned}
\]
Hence the sequence
$\{\Delta_{j,N}^{(k)}\}_{N\geq N_0}$
is strictly increasing and positive. Therefore
\[
\begin{aligned}
\Sigma_j(M_k)-\Sigma_{j+1}(M_k)
&=
\lim_{N\to\infty}
\Delta_{j,N}^{(k)}
\geq
\Delta_{j,N_0}^{(k)}>0.
\end{aligned}
\]
Thus
\[
\Sigma_j(M_k)>\Sigma_{j+1}(M_k)
\qquad
(j\geq1).
\]
Together with $\Sigma_0(M_k)-\Sigma_1(M_k)=1$, this proves
\eqref{eq:strict-monotonicity}.
\end{proof}

\begin{corollary}
\label{cor:core-multiplicity}
Every nonzero eigenvalue of $C_{M_k}$ is simple.
\end{corollary}

\begin{proof}
By Theorem~\ref{thm:strict-monotonicity}, $\Sigma_1(M_k)=\Sigma_0(M_k)-1$, and therefore
\[
\begin{aligned}
\|C_{M_k}\|_{\mathrm{H.S.}}^2
&=
\Sigma_0(M_k)+\Sigma_1(M_k)=
2\Sigma_0(M_k)-1=
2k^2\sum_{m=k}^{\infty}\frac1{m^2}-1=
1+
2k^2\sum_{m=k+1}^{\infty}\frac1{m^2}.
\end{aligned}
\]
On the other hand, Theorem~\ref{thm:core-spectrum} shows that
\[
1,\qquad
\pm\frac{k}{n+k},\quad n\geq1,
\]
are all the nonzero eigenvalues of $C_{M_k}$. Their squared values,
counted once each, satisfy
\[
\begin{aligned}
1+
2\sum_{n=1}^{\infty}
\left(\frac{k}{n+k}\right)^2
&=
1+
2k^2\sum_{m=k+1}^{\infty}\frac1{m^2}=
\|C_{M_k}\|_{\mathrm{H.S.}}^2.
\end{aligned}
\]
Since a compact self-adjoint Hilbert--Schmidt operator satisfies
\[
\|C_{M_k}\|_{\mathrm{H.S.}}^2
=
\sum_{\lambda\in\sigma(C_{M_k})}
m(\lambda)|\lambda|^2,
\]
where $m(\lambda)$ denotes the multiplicity of $\lambda$, the equality
above leaves no room for any additional multiplicity. Hence every
nonzero eigenvalue of $C_{M_k}$ is simple.
\end{proof}

\begin{remark}
\label{rem:finite-section-stronger}
The proof of Theorem \ref{thm:strict-monotonicity} actually gives a result stronger than the monotonicity of the
numerical invariant sequence. Namely, for every $j\geq1$,
\[
\sum_{n=0}^{N}
|\alpha_{n,j}^{(k)}|^2
>
\sum_{n=0}^{N}
|\alpha_{n,j+1}^{(k)}|^2
\]
for every $N\geq\max\{0,j-k\}$.
Thus the strict decrease is already present at every admissible finite
graded section; it is not produced merely by cancellation in the
infinite series.
\end{remark}

\begin{remark}
\label{rem:comparison-zk-wk}
The mechanism in Theorem \ref{thm:strict-monotonicity} is fundamentally different from that for the
submodules $[z^k-w^k]$ \cite{Liu2}. For $[z^k-w^k]$, the associated Toeplitz
matrices split into independent residue classes modulo $k$, and this
produces constant blocks in the numerical invariant sequence. For
$[(z-w)^k]$, no such residue-class decomposition occurs. Instead, the
adjacent numerical coefficients are connected by the first-order
relation \eqref{eq:adjacent-alpha}, and the corresponding finite
coefficient vectors are related by strict contractions. This produces
strict decrease rather than block repetition.
\end{remark}

\medskip

\subsection*{Acknowledgment}
Y. Liu was supported by the Natural Science Foundation Project in Henan Province (No. 262300421861), the funding program for young backbone teachers in higher education institutions in Henan Province (No. 2024GGJS106), the key research projects of higher education institutions in Henan Province (No. 25B110009) and the general project cultivation fund of Nanyang Normal University (No. 2025PY034), Y. Lu was supported by NNSFC (Grant No. 12031002), C. Zu was supported by NNSFC (Grant No. 12401151), and the Postdoctoral Researcher Foundation of China  (Grant No. GZB20240100).

%
\subsection*{Conflict of interest}
The authors declare that they have no conflict
of interest. 
\subsection*{Data availability statement}
No data, models, or code were generated or used for the research described in the article.

\end{document}